\documentclass[%
 reprint,
 amsmath,amssymb,
 aps,
]{revtex4-2}

\usepackage{graphicx}
\usepackage{dcolumn}
\usepackage{bm}
\usepackage{mathtools}
\usepackage{amsthm}
\usepackage{amsmath}
\usepackage{bbm}
\usepackage{bbold}

\newtheorem{theorem}{Theorem}
\newtheorem{lemma}{Lemma}
\newtheorem{proposition}{Proposition}
\newtheorem{corollary}{Corollary}

\newcommand{\naturals}{\mathbb{N}}
\newcommand{\integers}{\mathbb{Z}}

\newcommand{\probability}{\mathbb{P}}
\newcommand{\expectation}{\mathbb{E}}

\newcommand{\geom}{\text{Geom}}

\newcommand{\Uc}{\mathcal{U}}
\newcommand{\Rc}{\mathcal{R}}
\newcommand{\Tc}{\mathcal{T}}
\newcommand{\Vc}{\mathcal{V}}

\newcommand{\Xc}{\mathcal{X}}
\newcommand{\Yc}{\mathcal{Y}}

\newcommand{\one}{\mathbb{1}}

\begin{document}

\preprint{APS/123-QED}

\title{Using Restart Sequences to Determine Beneficial First Passage Under Restart}

\author{Jason M. Flynn}
 \email{jasonmflynn@ufl.edu}
\author{Sergei S. Pilyugin}%
 \email{pilyugin@ufl.edu}
\affiliation{%
 University of Florida,\\
 Department of Mathematics
}%

\date{\today}

\begin{abstract}
    First Passage (FP) processes are utilized widely to model phenomena in many areas of mathematical applications, from biology to computer science. Introducing a mechanism to restart the parent process can alter the first passage characteristics, and the properties of the resulting First Passage Under Restart (FPUR) process have been subject to much recent investigation. Here we seek to more fully characterize whether a discrete FP process can have its mean hitting time reduced through the implementation of restart by analyzing a construct we call the restart sequence, which is determined solely by the distribution of the FP process.
\end{abstract}

\maketitle


\section{The initial framework}
From \cite{flynn2021first}, we recall that for a given First Passage (FP) process with hitting time denoted by $\Uc$, we are interested in whether introducing a particular restart mechanism can produce a First Passage under Restart (FPUR) process with a reduced mean FP time (we call this case \underline{beneficial restart}). We define the FPUR restart process's hitting time, $\Tc$, recursively by
\begin{equation}\label{def:T}
    \Tc \coloneqq \begin{cases}
       \Uc & \text{if } \Rc > \Uc \\
       \Rc + \Tc^* & \text{if } \Rc \le \Uc
    \end{cases},
\end{equation}
where $\Tc^*$ is an iid copy of $\Tc$ and $\Rc$ represents our `restart time.' In general, $\Rc$ is a random variable that can be drawn from an arbitrary distribution. We must be aware, however, of the quantity $p_r \coloneqq \probability(\Rc\le\Uc)$, which denotes the probability that, on an individual foray, the FPUR process restarts before the underlying FP process finishes. When $p_r=1$ we see that the FPUR process will, with probability 1, restart before reaching its terminal state, meaning that $\Tc$ is not finite and that the FPUR process will never resolve. We call this \underline{preemptive restart}, and it is clearly a property of the interplay between the distributions of $\Uc$ and $\Rc$, rather than a property solely of $\Rc$.

There are multiple approaches to deriving formulas for $\expectation[\Tc]$ as seen in \cite{flynn2021first}, but perhaps the simplest can be derived directly from (\ref{def:T}) as follows.
\begin{align*}
    \expectation[\Tc] & = \expectation[\Uc \mid \Uc < \Rc]\cdot \probability(\Uc<\Rc) \\
    & \quad \quad + \expectation[\Rc + \Tc^* \mid \Rc \le \Uc]\cdot\probability(\Rc\le\Uc) \\
    & = \expectation[\Uc\mid\Uc<\Rc](1-p_r) + \expectation[\Rc\mid\Rc\le\Uc]p_r + \expectation[\Tc^*]p_r 
\end{align*}
Recalling that $\Tc^*$ is iid to $\Tc$, we can simply write
\begin{equation}\label{eq:ET_min}
    \expectation[\Tc] = \frac{\expectation[\Uc\wedge\Rc]}{1 - p_r}.
\end{equation}
This formulation lends itself to some immediate observations. First, the limit as $p_r\to1$ sends $\expectation[\Tc]\to\infty$ as previously mentioned (except in the case where $\Rc=0$). Second, if at least one of $\expectation[\Uc]$ and $\expectation[\Rc]$ is finite, then $\expectation[\Tc]<\infty$. That is, if we assume $\expectation[\Uc]=\infty$, we have $\expectation[\Tc]<\infty$ for any non-preemptive restart with $\expectation[\Rc]<\infty$, and thus trivially beneficial restart. We therefore limit our attention to cases where $\expectation[\Uc]<\infty$ when looking for beneficial restart.

%

\section{Introducing the restart sequence, $\{S_n\}$ }
We observe that restart is beneficial precisely when $\expectation[\Tc] - \expectation[\Uc] < 0$, or equivalently by (\ref{eq:ET_min}) as
\begin{align*}
    0 & > \expectation[\Tc] - \expectation[\Uc] \\
    & = \frac{\expectation[\Uc\wedge\Rc]}{1 - p_r} - \expectation[\Uc] \\
    & = \frac{1}{1-p_r}\left( \expectation[\Uc\wedge\Rc] - (1-p_r)\expectation[\Uc] \right) \\
    & = \frac{\sum_{n\ge0}\left( 1 - U(n) - u(n)\expectation[\Uc] \right)\left( 1 - R(n) \right)}{1-p_r}.
\end{align*}
Since we assume restart is non-preemptive, $1-p_r$ is clearly greater than 0, and we concern ourselves with the numerator.

For a given FP time, $\Uc$, with probability mass function (PMF), $u(n)$, and corresponding cumulative mass function (CMF), $U(n)$, we then define the ``restart sequence of $\Uc$'' by
\begin{equation}
    S_n^\Uc \coloneqq \sum_{k=0}^n \left( 1 - U(k) - u(k)\expectation[\Uc] \right),
\end{equation}
often abbreviated to simply $S_n$ when the FP process is unambiguous. Analysis of this sequence can provide substantial insight into the behavior of the FPUR process. We begin with the observation that for any distribution, $\Uc$, with $\expectation[\Uc]<\infty$, we have that $S_n^\Uc \to 0$ as $n \to \infty$ as seen below. Since $\expectation[\Uc] = \sum_{n\ge0}(1-U(n))$, we can write 
\begin{align*}
       \lim_{n\to\infty} S_n & = \lim_{n\to\infty} \sum_{k=0}^n \left( 1 - U(k) - u(k)\expectation[\Uc] \right) \\
       & = \lim_{n\to\infty} \sum_{k=0}^n \left( 1- U(k) \right) - \lim_{n\to\infty} \expectation[\Uc]U(n) \\
       & = \sum_{k\ge0} \left( 1- U(k) \right) - \expectation[\Uc]\lim_{n\to\infty}U(n) \\
       & = \expectation[\Uc] - \expectation[\Uc] = 0.
   \end{align*}

%

\subsection{Determining whether beneficial restart is possible}
\begin{lemma}
    Let $\{a_n\}$ be some sequence such that $\sum_{n\ge0}a_n<\infty$ and define $A_n \coloneqq \sum_{k=0}^n a_k$. Then the following are equivalent.
    \begin{itemize}
        \item $A_n\ge0$ for all $n\ge0$.
        \item $\sum_{n\ge0} a_nb_n \ge 0$ for any sequence, $\{b_n\}$, decreasing monotonically to 0.
    \end{itemize}
\end{lemma}

\begin{theorem} \label{th:Sn_restart}
   $\Uc$ permits beneficial restart iff there exists some $n\ge0$ such that $S_n<0$.
\end{theorem}

\begin{proof}
   We note that the quantity $\expectation[\Tc]-\expectation[\Uc]$ has the same sign as $\sum_{n\ge0}\left( 1 - U(n) - u(n)\expectation[\Uc] \right)\left( 1 - R(n) \right)$. Clearly $b_n\coloneqq1-R(n)$ is monotonically decreasing to 0, so by our Lemma, restart is non-beneficial iff $\sum_{k=0}^n\left( 1 - U(k) - u(k)\expectation[\Uc] \right) \ge 0$.
\end{proof}

\begin{corollary}
   If there exists a beneficial restart, then there exists a beneficial sharp restart. Furthermore, sharp restart will minimize the value of $\expectation[\Tc]$ across all restart distributions.
\end{corollary}
\begin{proof}
   Suppose there there exists some restart distribution, $R(n)$, for which $\Uc$ permits a beneficial restart. Then we know there exists some $n^*\in\naturals$ such that $S_{n^*}<0$. Let $\tilde R(n)$ be the CMF of $\tilde \Rc = n^*+1$, that is, $\tilde R(n) = \one_{[n^*+1,\infty)}(n)$. Looking at the proof of Theorem \ref{th:Sn_restart}, we can immediately see that this distribution will also provide a beneficial restart, and thus sharp restart at $N=n^*+1$ is beneficial.
   
   For the second part of the claim, we first make an observation. Let $x=\{x_n\}$ be an element of $X$, the set of nonnegative summable sequences, and consider $F(x) \coloneqq \frac{\sum_n\alpha_nx_n}{\sum_n\beta_nx_n}$ where $\beta_n > 0$, with the intent to minimize $F(x)$ over $X$. It's clear that for any number $\lambda$ such that $\alpha \ge \lambda\beta$, we have 
   \begin{equation*}
       F(x) = \frac{\sum_n\alpha_nx_n}{\sum_n\beta_nx_n} \ge \frac{\sum_n\lambda\beta_nx_n}{\sum_n\beta_nx_n} = \lambda.
   \end{equation*}
   Now, supposing there exists a limit $L \coloneqq \lim_{n\to\infty}\frac{\alpha_n}{\beta_n}$, we have two possibilities:
   \begin{itemize}
       \item if $\frac{\alpha_n}{\beta_n}>L$ for all $n$, then $F(x)>L$ for all $x$ and $\inf_x F(x) = L$, but the minimum is never attained, or
       \item if $\frac{\alpha_n}{\beta_n} \le L$ for some $n$, then there exists $m$ such that $\min_n \frac{\alpha_n}{\beta_n} = \frac{\alpha_m}{\beta_m}$. Furthermore, for $e_m \coloneqq \{\delta_{m,n}\}_{n\ge0}$, we have that $\min_x F(x) = F(e_m) = \frac{\alpha_m}{\beta_m}$, and thus $F$ attains its minimum value.
   \end{itemize}
   We now apply this observation to equation (\ref{eq:ET_min}) where
   \begin{align*}
       \expectation[\Uc\wedge\Rc] & = \sum_{n}(1-U(n))(1-R(n)) \\
       & = \sum_{n\ge0}(1-U(n))\sum_{k>n}r(k) \\
       & = \sum_{k>0}r(k)\sum_{n=0}^{k-1}(1-U(n)) \\
       & \text{and} \\
       1 - p_r & = \sum_n u(n)(1-R(n)) = \sum_{n\ge0}u(n)\sum_{k>n}r(k) \\
       & = \sum_{k>0}r(k)\sum_{n=0}^{k-1}u(n) = \sum_{k>0}r(k)U(k-1).
   \end{align*}
   Letting $\alpha_k \coloneqq \sum_{n=0}^{k-1}(1-U(n))$ and $\beta_k \coloneqq U(k-1)$, we have $\beta_k>0$. We can then rearrange 
   \begin{equation*}
       S_{k-1} = \sum_{n=0}^{k-1}(1-U(n)) - U(k-1)\expectation[\Uc] = \alpha_k - \expectation[\Uc]\beta_k
   \end{equation*}
   to get 
   \begin{equation*}
       \frac{\alpha_k}{\beta_k} = \expectation[\Uc] + \frac{S_{k-1}}{\beta_k} = \expectation[\Uc] + \frac{S_{k-1}}{U(k-1)}.
   \end{equation*}
   Since $S_{k-1}\to0$ and $U(k-1)\to 1$, we have $\lim_{k\to\infty}\frac{\alpha_k}{\beta_k} = \expectation[\Uc]$. For a sequence $r=\{r(n)\}_{n\ge0}$, we have that 
   \begin{equation*}
       F(r) = \frac{\sum_n \alpha_n r(n)}{\sum_n \beta_n r(n)} = \expectation[\Tc],
   \end{equation*}
   which gives us two possibilities:
   \begin{itemize}
       \item if $S_k > 0$ for all $k$, then $\frac{\alpha_k}{\beta_k} > \expectation[\Uc]$ and thus $\expectation[\Tc]=F(r)>\expectation[\Uc]$ for all summable sequences $r$, and no restart is beneficial, or
       \item if $S_k \le 0$ for some $k$, then there exists $m \ge 1$ such that 
       \begin{equation*}
           \min_k \frac{\alpha_k}{\beta_k} = \frac{\alpha_m}{\beta_m} = \expectation[\Uc] + \frac{S_{m-1}}{U(m-1)} \le \expectation[\Uc].
       \end{equation*}
       In this case, sharp restart at $m$, that is, $r = e_m$ will minimize $\expectation[\Tc] = F(e_m) \le \expectation[\Uc]$. In particular, if $S_k$ is ever negative, then $\expectation[\Tc] = F(e_m)<\expectation[\Uc]$, and there exists a most beneficial restart, which is sharp.
   \end{itemize}
\end{proof}

To illustrate an immediate application of this theorem, we compute the sequence $\{S_n\}$ for some simple distributions below to check for beneficial restart.

%

\subsubsection{Example: $\Uc = N$}
For the deterministic FP time, we have $u(n) = \delta_{n,N}$, $U(n) = \one_{[N,\infty)}(n)$ and $\expectation[\Uc] = N$. Thus we can directly compute the sequence, $\{S_n\}$.
\begin{align*}
    S_n & = \sum_{k=0}^n \left( 1 - U(k) - u(k)\expectation[\Uc] \right) \\
    & = \sum_{k=0}^n \left( 1 - \one_{[N,\infty)}(k) - \delta_{k,N}N \right) \\
    & = \sum_{k=0}^n \left( \one_{[0,N)}(k) - \delta_{k,N}N \right) \\
    & = 
    \begin{cases}
       n+1 & \text{if } 0 \le n < N\\
       0 & \text{otherwise}
    \end{cases}
\end{align*}
Note that $S_n \ge 0$ for all $n$, and thus a deterministic FP time cannot be reduced by adding restart.

%

\subsubsection{Example: $\Uc \sim \text{Geom}(\rho)$}
For the geometrically distributed FP time, we have that $u(n) = \rho(1-\rho)^n$, $U(n) = 1 - (1-\rho)^{n+1}$ and $\expectation[\Uc] = \frac{1-\rho}{\rho}$. Thus we can directly compute the sequence, $\{S_n\}$.
\begin{align*}
    S_n & = \sum_{k=0}^n \left( 1 - U(k) - u(k)\expectation[\Uc] \right) \\
    & = \sum_{k=0}^n \left( 1 - (1 - (1-\rho)^{k+1}) - \rho(1-\rho)^k\cdot\frac{1-\rho}{\rho} \right) \\
    & = \sum_{k=0}^n \left( (1-\rho)^{k+1} - (1-\rho)^{k+1} \right) \\
    & = 0
\end{align*}
Note that since $S_n \equiv 0$, it is trivially greater than or equal to $0$ and $\Uc$ does not permit beneficial restart. In fact, the geometric distribution is, in some sense, the threshold of beneficial restart, as we shall see throughout the paper. Furthermore, for any non-preemptive restart, $\expectation[\Tc] = \expectation[\Uc]$; that is, introducing restart does not change the mean hitting time of a geometrically distributed FP process.

%

\subsubsection{Example: $\Uc$ as the mixture of two point masses}\label{ex:two_point_mixture}
For weights $w_1,w_2\in(0,1)$ and $M_1,M_2\in\naturals$ with $w_1+w_2=1$ and $M_2>M_1$, we define the PMF $u(n) = w_1\delta_{n,M_1} + w_2\delta_{n,M_2}$. We can easily compute the cumulative mass function and $\expectation[\Uc]$, so we can write
\begin{equation*}
    S_n = 
    \begin{cases}
       n+1 & n<M_1 \\
       w_2(n+1) - w_1w_2(M_2-M_1) & M_1 \le n < M_2 \\
       0 & M_2 \le n
    \end{cases}.
\end{equation*}
It's clear that $S_n$ will be nonnegative outside of $[M_1,M_2)$, but $S_{M_1}<0$ exactly when $M_2>\frac{(w_1+1)M_1+1}{w_1}$. That is, when the two point masses are sufficiently far apart, sharp restart at $M_1+1$ will be beneficial.

%

\section{Properties and techniques}
Since the behavior of the restart sequence completely determines which distributions permit beneficial restart, it's helpful to understand some of its properties and to develop some tools to help examine it more closely. To that end, we introduce some notation and terminology below.

%

\subsection{First Step Analysis (FSA)}
We begin simply, by defining $a \coloneqq \min\{n\mid u(n)>0\}$ and $b \coloneqq \max\{n\mid u(n)>0\}$, if such a maximum exists. Note that $a\ge0$, and $b \ge a$ exists only when the distribution of $\Uc$ is finitely supported. Furthermore, $a>0$ has an interesting and immediate result: since $u(n)$ and $U(n)$ are 0 for all $n<a$, we can see that $S_n = n+1$ for all $n<a$. Thus, the first possibility for $S_n<0$ (and therefore beneficial restart) is at $n=a$. Thus, we can check $S_a<0$, which if true indicates that a beneficial restart exists. The converse, however, is not true. In the other case, when $S_a \ge 0$, $S_n<0$ could be true for some larger $n$ and thus beneficial restart might still be possible. Ultimately $S_a < 0$ is sufficient, but not necessary, to show the existence of a beneficial restart and, crucially, usually easy to check.
\begin{proposition}\label{prop:FSA}
   If $\frac{a+1-u(a)}{u(a)} < \expectation[\Uc]$, then $\Uc$ admits beneficial restart.
\end{proposition}
We call checking the inequality from this proposition the ``First Step Analysis," or FSA.
\begin{proof}
   We know that $S_{a-1} = a$ for $a>0$, hence 
   \begin{align*}
       S_a & = S_{a-1} + 1 - U(a) - u(a)\expectation[\Uc] \\
       & = a + 1 - u(a) - u(a)\expectation[\Uc]).
   \end{align*}
   This quantity is negative precisely when
   \begin{equation*}
       u(a)\expectation[\Uc] > a + 1 - u(a), \text{ or}
    \end{equation*}
    \begin{equation} \label{eq:FSA}
        \expectation[\Uc] > \frac{a+1-u(a)}{u(a)}.
    \end{equation}
   Clearly, the same holds for $a=0$.
\end{proof}

%

\subsection{Finitely supported distributions}
When $b < \infty$, we can say even more. Since the support of $\Uc$ is contained in $\{a, \ldots, b\}$, we know that $U(n) = 1$ for $n \ge b$ and $u(n) = 0$ for $n > b$. Thus, for any $n > b$, we can see that 
\begin{align*}
    S_n & = S_b + \sum_{k=b+1}^n (1 - U(k) - u(k)\expectation[\Uc]) \\
    & = S_b + \sum_{k=b+1}^n (1 - 1 - 0) = S_b,
\end{align*}
and thus the restart sequence becomes constant after $b$. Notice, however, that $S_b = \sum_{k=0}^b (1 - U(k)) - U(b)\expectation[\Uc]$ which is exactly 0. Therefore, we can see that $S_n=0$ for all $n \ge b$. In fact, the sequence will always descend to 0 at $b$, as we show next.
\begin{proposition}\label{prop-finite_S_descend}
   For a finitely supported distribution, $\Uc$, the last non-zero step of the sequence $S_n^\Uc$ will be down. That is, $S_{b-1}^\Uc>0$.
\end{proposition}
\begin{proof}
   Let's examine the difference $S_b - S_{b-1}$. We can clearly see that 
   \begin{align*}
       S_b - S_{b-1} & = S_{b-1} + (1 - U(b) - u(b)\expectation[\Uc]) - S_{b-1}\\
       & = 1 - 1 - u(b)\expectation[\Uc] \\
       & = - u(b)\expectation[\Uc].
   \end{align*}
   Since $u(b)>0$ by the definition of $b$ and $S_b = 0$, we can immediately see that $S_{b-1} = u(b)\expectation[\Uc]>0$.
\end{proof}

%

\subsection{Forcing a gap in the support of $\Uc$}
It's natural to ask how various perturbations to the distribution of $\Uc$ will affect $S_n^\Uc$, and the first we consider is the introduction of a gap into the support, in which a distribution is cut at a certain point and all subsequent probabilistic mass is shifted back. That is, let $g,m$ be chosen from $\naturals$. Then define $\tilde\Uc$ with the following distribution:
\begin{equation*}
    \tilde u(n) = 
    \begin{cases}
       u(n) & \text{ if } n < m \\
       0 & \text{ if } m \le n < m+g \\
       u(n-g) & \text{ if } m+g \le n
    \end{cases}.
\end{equation*}
Clearly, $\tilde u(n)$ is still a probability distribution, and we can compute its mean to be $\expectation[\tilde \Uc] = \expectation[\Uc] + g(1-U(m-1))$ for $m>0$ (simply $\expectation[\Uc] + g$ for $m=0$). The cumulative mass function for $\tilde \Uc$ can similarly be written in terms of $\Uc$.
\begin{equation*}
    \tilde U(n) =
    \begin{cases}
       U(n) & \text{ if } n < m \\
       U(m-1) & \text{ if } m \le n < m+g \\
       U(n-g) & \text{ if } m+g \le n
    \end{cases}
\end{equation*}

Now, we may compute $S_n^{\tilde \Uc}$. We consider three regions, just as with the distribution above. 
\begin{equation*}
    S_n^{\tilde\Uc} = 
    \begin{cases}
       S_n^\Uc - U(n)(1-U(m-1))g \\
       \quad\quad\quad \text{ if } n < m \\
       S_{m-1}^\Uc + (1-U(m-1))(1 - U(m-1)g + n-m) \\
       \quad\quad\quad \text{ if } m \le n < m+g \\
       S_{n-g}^\Uc + (1-U(m-1))(1-U(n-g))g \\
       \quad\quad\quad \text{ if } m+g \le n
    \end{cases}
\end{equation*}
\begin{itemize}
    \item \textbf{Pre-Gap:} If $n<m$, it is clear that $S_n^{\tilde\Uc}\le S_n^\Uc$ with equality only if the support hasn't started by, or finishes before, the gap. Thus, if $\Uc$ has a beneficial restart before the gap, then so will $\tilde\Uc$. If $\Uc$ does not permit beneficial restart, then $\tilde\Uc$ will for $g>\frac{S_n^\Uc}{U(n)(1-U(m-1))}$. This is true provided $U(m-1)<1$, which is always the case if $\Uc$ has unbounded support (like the geometric). This makes sense considering that $\lim_{g\to\infty}\expectation[\tilde\Uc]=\infty$, and we know that non-preemptive restart is beneficial in that case.
    \item \textbf{Within-Gap:} For $m \le n < m+g$, we can see that $S_n^{\tilde\Uc}$ is monotonically increasing. On this interval, therefore, we have that 
    \begin{equation*}
         1 - gU(m-1) \le \frac{S_n^{\tilde\Uc} - S_{m-1}^\Uc}{1-U(m-1)} \le (1-U(m-1))g.
    \end{equation*}
    So long as $g\le\frac{1}{U(m-1)}$ then $S_n^{\tilde\Uc}\ge S_{m-1}^\Uc$, and a lack of beneficial restart at $m$ means that $\tilde\Uc$ won't permit beneficial restart within the gap. If, however, $g\ge\frac{1}{U(m-1)}$ or $S_{m-1}^\Uc<0$, then it's not immediately clear.
    \item \textbf{Post-Gap:} In the last case, where $n \ge m+g$, we see that $S_n^{\tilde\Uc} \ge S_{n-g}^\Uc$. Thus $S_n^\Uc\ge0$ implies that restart won't be beneficial after the support of $\tilde\Uc$ resumes. On the other hand, if restart is beneficial for $\Uc$ at $N+1$, then it will be for $\tilde\Uc$ at $N+1+g$ iff $g<\frac{|S_N^\Uc|}{(1-U(m-1))(1-U(N))}$.
\end{itemize}
An interesting subcase is the one in which the entire distribution is pushed back, which we consider now.

%

\subsubsection{Delaying a FP time}
If $m=0$, then the previous trichotomy reduces to 
\begin{equation*}
    S_n^{\tilde\Uc} = 
    \begin{cases}
       n + 1  & \text{ if } 0 \le n < g \\
       S_{n-g}^\Uc + (1 - U(n-g))g & \text{ if } g \le n
    \end{cases}.
\end{equation*}
Thus, we can say that $S_n^{\tilde\Uc}\ge0$ for $n < g$ and $S_n^{\tilde\Uc}\ge S_{n-g}^\Uc$ for $n \ge g$. Thus, if $\Uc$ does not permit beneficial restart, neither will $\tilde\Uc$. This sheds more light on our previous example with the shifted geometric distribution; delaying by 1 only increased the values of $S_n$.

If, on the other hand, sharp restart at $N+1$ is beneficial for $\Uc$, we can see that restart at $N+1+g$ will be beneficial for $\tilde\Uc$ iff $g < \frac{|S_N^\Uc|}{1-U(N)}$. Equivalently, we can say if the gap is too large, then $\tilde\Uc$ will not permit beneficial restart, even if $\Uc$ does.

%

\subsection{Uniqueness of the distribution}\label{sec:uniqueness_of_Sn}
It is clear that each distribution, $\Uc$, will generate a unique restart sequence, so we are immediately prompted to ask whether each restart sequence is associated with a single distribution. Suppose that $\Uc$ and $\Vc$ are distinct distributions, but that $\{S_n^\Uc\}$ and $\{S_n^\Vc\}$ are the same sequence.
\begin{align*}
    0 & = S_n^\Uc - S_n^\Vc \\
    & = \sum_{k=0}^n(1 - U(k) - u(k)\expectation[\Uc]) - \sum_{k=0}^n(1 - V(k) - v(k)\expectation[\Vc]) \\
    & = \sum_{k=0}^n (V(k)-U(k) + v(k)\expectation[\Vc] - u(k)\expectation[\Uc])
\end{align*}
Since this must be true for all $n\ge0$, we have 
\begin{equation*}
    0 = V(k) + v(k)\expectation[\Vc] - U(k) - u(k)\expectation[\Uc]
\end{equation*}
for all $k\ge0$. When $k=0$, this equates to $0 = v(0)(1+\expectation[\Vc]) - u(0)(1+\expectation[\Uc])$. If we suppose further that $\Uc$ and $\Vc$ have the same mean, then it's clear that $\Uc$ and $\Vc$ must have the same distribution. Without this last supposition, however, we can expect no such result.

In general, define $p \coloneqq \frac{1}{\expectation[\Uc]+1}$ and $q \coloneqq 1-p$. We know that $U(0) = \frac{1-S_0}{\expectation[\Uc]+1}$, and since $S_n = S_{n-1} + 1 - U(n) + \expectation[\Uc](U(n)-U(n-1))$ for $n\ge1$, we can write 
\begin{equation*}
    U(n) = p(1+S_{n-1}-S_n) + qU(n-1).
\end{equation*}
Defining $F_n = 1+S_{n-1}-S_n$ and $S_{-1} = 0$, we have
\begin{equation}\label{eq:conv}
    U(n) = p\sum_{k=0}^nF_{n-k}q^k.
\end{equation}
That is, since $F_n$ is the same for $\Uc$ and $\Vc$, and $p,q$ depend only on the mean of $\Uc$ or $\Vc$, we offer the following proposition.
\begin{proposition}
   If a given a restart sequence, $\{S_n\}$, can be associated with a distribution, $\Uc$, such distribution is determined uniquely up to $\expectation[\Uc]$.
\end{proposition}

\subsection{Convexity analysis on the restart sequence}
We begin by defining convexity of the restart sequence in a standard way: let $C_n \coloneqq \frac{S_{n-1} - 2S_n + S_{n+1}}{2}$ and say that the sequence is convex at $n$ if $C_n>0$, concave at $n$ if $C_n<0$ or linear at $n$ if $C_n=0$. In fact, there is a quick way to check this inequality.
\begin{proposition}
   Define $K \coloneqq \frac{\expectation[\Uc]}{\expectation[\Uc] + 1}$.
   \begin{itemize}
       \item $\{S_n\}$ is convex at $n$ if $u(n+1) < Ku(n)$,
       \item $\{S_n\}$ is linear at $n$ if $u(n+1) = Ku(n)$, and
       \item $\{S_n\}$ is concave at $n$ if $u(n+1) > Ku(n)$.
   \end{itemize}
\end{proposition}
\begin{proof}
   \begin{align*}
       C_n & = \frac{S_{n-1} - 2S_n + S_{n+1}}{2} \\
       & = \frac{1}{2}\left( S_{n-1} - 2(S_{n-1} + 1 - U(n) - u(n)\expectation[\Uc]) \right.\\
       & \quad\quad + (S_{n-1} + 1 - U(n) - u(n)\expectation[\Uc] + 1 \\
       & \quad\quad \left.- U(n+1) - u(n+1)\expectation[\Uc]) \right) \\
       & = \frac{1}{2}\left( U(n) - U(n+1) + (u(n)-u(n+1))\expectation[\Uc] \right) \\
       & = \frac{1}{2}\left( u(n)\expectation[\Uc] - u(n+1)(1 + \expectation[\Uc]) \right) \\
       & = \frac{\expectation[\Uc]+1}{2}\left( u(n)K - u(n+1) \right)
   \end{align*}
\end{proof}

As a note, it can be useful to consider the ratio $r_n \coloneqq \frac{u(n+1)}{u(n)}$ when appropriate. In this case, we need only compare $r_n$ and $K$ to determine convexity at $n$.

\subsubsection{Bounded distributions}
For distributions with finite support that produce restart sequences with constant convexity, a lot can be said. Since $\{S_n\}$ decreases to 0 at $b$ by Proposition \ref{prop-finite_S_descend}, we can immediately make the following proposition.
\begin{proposition}
   Let $\Uc$ have a distribution with finite support. Then
   \begin{itemize}
       \item $\{S_n\}$ convex or linear on its support means $S_n\ge0$ for all $n$, and
       \item $\{S_n\}$ concave on its support means that $S_a<0$ iff $S_n<0$, and we need only perform a FSA by checking inequality (\ref{eq:FSA}).
   \end{itemize}
\end{proposition}
We note here that any that any non-decreasing PMF will produce a concave $\{S_n\}$, since $K<1$. For example, we can then immediately confirm that a constant distribution does not permit beneficial restart. Since $u(n)=\frac{1}{b-a+1}$ for all $a \le n \le b$, we have $\frac{u(n+1)}{u(n)} = 1 > K$ and $\{S_n\}$ is concave. We then check inequality (\ref{eq:FSA}) and see 
\begin{align*}
    \frac{a+1-u(a)}{u(a)} & = \frac{a + 1 - \frac{1}{b-a+1}}{\frac{1}{b-a+1}} \\
    & = (b-a+1)(a+1)-1 \\
    & = ab - a^2 + a + b - a + 1 - 1 \\
    & = a(b-a) + b \ge b \ge \expectation[\Uc].
\end{align*}
Thus $\Uc\sim\text{Unif}(a,b)$ does not permit beneficial restart.

\subsection{Convex or concave tails}
In cases where $\Uc$ has infinite support, there are additional considerations. Rather than considering the entire support, we often just consider the tail of $\{S_n\}$. Since $S_n\to0$, we can make the following observation.
\begin{proposition}
   Suppose there exists some $M\in\naturals$ such that the sign of $d_n \coloneqq u(n+1)-Ku(n)$ is constant for all $n \ge M$. Then the tail of $\{S_n\}$ is either convex ($d_n<0$), concave ($d_n>0$), or linear ($d_n=0$) after $M$. In the case of a convex or linear tail, $S_n\ge0$ for all $n \ge M$ and sharp restart after $M$ won't be beneficial. For a concave tail, sharp restart at \underline{any} time after $M$ will be beneficial.
\end{proposition}

\subsubsection{Example of a convex tail: the shifted geometric}
We've already seen the so-called ``shifted geometric,'' or the geometric distribution on $\integers^+$ with rate parameter $\rho$. This clearly satisfies $u(n+1)-Ku(n)<0$ for $M=1$, since $K=\frac{1}{1+\rho}$. As previously determined, this distribution does not allow for beneficial restart.
\subsubsection{Examples of a linear tail}
The only linear subsequence that converges to zero is the one with all zeros, and that's what we see in the tails below. 
\begin{itemize}
    \item The standard geometric distribution on $\naturals$ produces the restart sequence with $S_n=0$ for all $n$.
    \item Any finitely supported distribution has $S_n=0$ for all $n \ge b$. 
    \item As discussed in Section \ref{sec:uniqueness_of_Sn}, the above case need not arise only from a finitely supported distribution. The sequence $\{S_n\}=\{1,0,0,\ldots\}$ can clearly arise as a result of $\Uc = 1$ when $\expectation[\Uc]=1$. But more generally, we can produce a distribution associated with this restart sequence using formula (\ref{eq:conv}) for other values of $\expectation[\Uc]$. The sequence $\{F_n\}$ becomes $\{0,2,1,1,\ldots\}$, which allows for $U(n)=p\sum_{k=0}^nF_{n-k}q^k$ with $p=\frac{1}{\expectation[\Uc]+1}$ and $q = \frac{\expectation[\Uc]}{\expectation[\Uc]+1} = K$, and thus $u(0)=0$, $u(1)=2p$ and $u(n)=(1-2p)pq^{n-2}$ for $n\ge2$ with $\expectation[\Uc]\ge1$. Since $q=K$ from our earlier definition, it's clear that, for $n\ge2$, we have $\frac{u(n+1)}{u(n)} = K$, and thus the tail of $\{S_n\}$ is linear (all zeros).
\end{itemize}
\subsubsection{Example of a concave tail: the zeta distribution on $\naturals$}
For $p\in(1,\infty)$, define $u(n) = (\zeta(p)(n+1)^p)^{-1}$ for $n \ge 0$, which has mean $\expectation[\Uc] = \frac{\zeta(p-1)}{\zeta(p)}-1$. Then, for $p>2$, $K<1$ is fixed and $r_n = \frac{(n+1)^p}{(n+2)^p}$ is monotonically increasing to 1. Thus there exists some $M\in\naturals$ such that $r_n>K$ for all $n>M$, and sharp restart is beneficial after this point (and possibly before).

\subsection{Log-concave distributions}
We begin by recalling that a non-negative sequence, $\{a_n\},$ is called \underline{logarithmically-concave} or just \underline{log-concave} if $a_n^2 \ge a_{n-1}a_{n+1}$ and the sequence has no internal 0s. We now consider FP times that have log-concave PMF. 

On the support of $\Uc$, it is useful to define $r_n = \begin{cases}\frac{u(n+1)}{u(n)} & a \le n \le b \\ 0 & \text{otherwise}\end{cases}$, which is a weakly decreasing sequence in $n$ for $u(n)$ log-concave.

\begin{proposition} \label{prop:log-concave}
   Let $\Uc$ be a FP process with a log-concave PMF. Then $\Uc$ does not permit beneficial restart.
\end{proposition}

\begin{proof}
   Suppose by way of contradiction that there exists some $n^*$ such that $S_{n^*}<0$ and beneficial restart is possible. Since $\{S_n\}$ converges to $0$, there exists some $s \coloneqq \min_n\{S_n\}$, and we define $m \coloneqq \min\{n \mid S_n = s\}$. We have immediately that $S_m \le S_{m+1}$, and then there are two cases to address.
   \begin{itemize}
       \item $m > 0$ \\ In this case, we also have that $S_m < S_{m-1}$, and thus we can write 
       \begin{align*}
           S_m - S_{m-1} & < 0 \le S_{m+1} - S_m \\
           1 - U(m) - u(m)\expectation[\Uc] & < 1 - U(m) - u(m+1)\left(\expectation[\Uc] + 1\right)
        \end{align*}
        \begin{equation*}
           K = \frac{\expectation[\Uc]}{\expectation[\Uc] + 1} > \frac{u(m+1)}{u(m)} \ge \frac{u(n+1)}{u(n)},
       \end{equation*}
       for all $n \ge m$.
       \item $m = 0$ \\ In this case, we have instead that $S_m < 0$, and thus we can write 
       \begin{align*}
           S_m & < 0 \le S_{m+1} - S_m \\
           1 - U(0) - u(0)\expectation[\Uc] & < 1 - U(0) - u(1)\left(\expectation[\Uc] + 1\right)
        \end{align*}
        \begin{equation*}
           K = \frac{\expectation[\Uc]}{\expectation[\Uc] + 1} > \frac{u(1)}{u(0)} \ge \frac{u(n+1)}{u(n)},
       \end{equation*}
       for all $n \ge m$.
   \end{itemize}
   In either case, the restart sequence in convex after $m$. Whether the support is finite or infinite, this is not possible, and thus $\Uc$ does not permit beneficial restart.
\end{proof}

\subsubsection{Example: Binomial}
For $N\in\integers^+$ and $p\in[0,1]$, we define $u(n) = \binom{N}{n}p^n(1-p)^{N-n}$. We can then compute $r_n = \frac{p(N-n)}{(1-p)(n+1)}$, which is decreasing in $n$ over its support and therefore $\Uc$ has a log-concave PMF and does not permit beneficial restart.

\subsubsection{Example: Poisson}
For $\lambda\in(0,\infty)$, we define $u(n) = \frac{\lambda^n e^{-\lambda}}{n!}$, which gives us $r_n = \frac{\lambda}{n}$. Since $r_n$ is clearly decreasing in $n$, we have that the PMF of $\Uc$ is log-concave, and $\Uc$ does not permit beneficial restart.

\subsection{Log-convex distributions}
Having identified such a definitive result for FP procesess with log-concave PMF, it is reasonable to ask what can be said about those with log-convex PMF, that is, those that satisfy the inequality $u(n)^2 \le u(n-1)u(n+1)$. These distributions have a different nature than those that are log-concave, and aren't as amenable to interpretations with finite support. For those with infinite support, however, we offer the following proposition.

\begin{proposition} \label{prop:log-convex}
   Let $\Uc$ be a FP process with a log-convex PMF. Then either $\Uc$ permits beneficial restart or $\Uc\sim\geom(\rho)$.
\end{proposition}

\begin{proof}
   The proof follows the same form as that of Proposition \ref{prop:log-concave}. Simply assuming that there exists some $n^*$ such that $S_{n^*} \ge 0$ and reversing all the inequalities (paying special attention to which inequalities are strict) leads to the following conclusion: either $S_n = 0$ for all $n \ge 0$ (with $r_n$ constant) or we have a contradiction. Thus, we have the result. As a special note, if $\Uc$ is not geometrically distributed, then $S_n < 0$ for all $n \ge 0$, and all non-preemptive sharp restarts are beneficial. 
\end{proof}

\subsubsection{Example: Zeta distribution supported on $\naturals$}
If we consider our previous example again, we can see that the zeta distribution on $\{0,1,2,\ldots\}$ is log-convex. Since $r_n$ is non-constant, we have $S_n < 0$ for all $n \ge 0$ and all non-preemptive sharp restarts are beneficial.

\subsubsection{Example: Negative Binomial}
For $r\in\integers^+$ and $p\in[0,1]$, we define $u(n) = \binom{n+r-1}{n}p^n(1-p)^r$. We can then compute $r_n = \frac{p(n+r)}{(n+1)}$, which is weakly decreasing in $n$ over its support and therefore $\Uc$ has a log-concave PMF and doesn't permit beneficial restart. In particular, if $r = 1$, then $u(n)$ reduces to $(1-p)p^n$, the geometric distribution.

While the standard model with $r\in\integers^+$ has a very natural interpretation, of some interest is the case where we relax that condition and instead let $r$ take values in $(0,\infty)$. It's not precisely clear what this process is modeling, but it is still a well-defined probability distribution. And, as soon as $r$ passes below 1, we see that $r_n$ becomes monotonically increasing and thus $u(n)$ is log-convex.

In summary, for $r \ge 1$, beneficial restart is impossible, but all non-preemptive sharp restarts become beneficial with $r < 1$.

\section{Mixture Distributions}
Considering the restart sequence generated by a mixture distribution can allow us to take a different angle for approaching some distributions that are more easily understood as the convex linear combination of possibly simpler distributions, and also to explore perturbations of known distributions.

Let's consider a case when distribution of our FP process is the convex linear combination of two other distributions; that is, we have FP time $\Yc$ with PMF given by $y(n) = w_1x_1(n) + w_2x_2(n)$ with $w_1 + w_2 = 1$ and $x_1(n), x_2(n)$ the PMF of two independent FP times $\Xc_1$ and $\Xc_2$. We can immediately see that the cumulative distributions have the same pattern, $Y(n) = w_1X_1(n) + w_2X_2(n)$, and as a result, we have that $\expectation[\Yc] = w_1\expectation[\Xc_1] + w_2\expectation[\Xc_2]$. Notice that we are not saying $\Yc = w_1\Xc_1 + w_2\Xc_2$. $\Yc$ is not the blending of two random variables, but rather a random variable drawn from a blended distribution.

Our interest is now in determining whether this new distribution will permit beneficial restart. To this end we compute the partial sums, $S_n^\Yc$, given the sums $S_n^{\Xc_1}$ and $S_n^{\Xc_2}$. We then consider the circumstances (if any) under which $S_n^\Yc < 0$.
\begin{align*}
    S_n^\Yc & = \sum_{k=0}^n\left[ 1 - Y(k) - \expectation[\Yc]y(k) \right] \\
    & = \sum_{k=0}^n \left[(w_1 + w_2) - (w_1X_1(k) + w_2X_2(k)) \right. \\
    & \quad \quad \quad - (w_1\expectation[\Xc_1]x_1(k) + w_2\expectation[\Xc_2]x_2(k)) \\
    & \quad \quad \quad + (w_1\expectation[\Xc_1]x_1(k) + w_2\expectation[\Xc_2]x_2(k)) \\
    & \quad \quad \quad - \left. (w_1\expectation[\Xc_1] + w_2\expectation[\Xc_2])(w_1x_1(k) + w_2x_2(k))\right] \\
    & = w_1S_n^{\Xc_1} + w_2S_n^{\Xc_2} \\
    & \quad \quad \quad + w_1w_2\sum_{k=0}^n(\expectation[\Xc_1] - \expectation[\Xc_2])(x_1(k) - x_2(k))
\end{align*}
Computing the last sum gives us the formula we desire:
\begin{equation} \label{eq:S_n^Y}
    \scriptsize{
    S_n^\Yc = w_1S_n^{\Xc_1} + w_2S_n^{\Xc_2} + w_1w_2(\expectation[\Xc_1] - \expectation[\Xc_2])(X_1(n)-X_2(n)).}
\end{equation}
From (\ref{eq:S_n^Y}) we can make several observations. Among them we have:
\begin{itemize}
    \item Suppose $\expectation[\Xc_1] = \expectation[\Xc_2]$. Then $S_n^\Yc$ reduces to $w_1S_n^{\Xc_1} + w_2S_n^{\Xc_2}$. In the case where $S_n^{\Xc_1},S_n^{\Xc_2} \ge 0$, then so is $S_n^\Yc$, and $\Yc$ does not permit beneficial restart. In the opposite case, where $S_n^{\Xc_1},S_n^{\Xc_2}<0$, again $S_n^\Yc$ follows and beneficial restart is possible.
    \item Without loss of generality, suppose that $\expectation[\Xc_1] > \expectation[\Xc_2]$. Equivalently, 
    \begin{align*}
        \sum_{n\ge0} (1 - X_1(n)) - (1 - X_2(n)) & > 0 \\
        \sum_{n\ge0} (X_2(n) - X_1(n)) > 0.
    \end{align*}
    Thus we know $X_1(n) - X_2(n) < 0$ for at least some $n\ge0$, which implies that the last term in (\ref{eq:S_n^Y}) must be negative for some $n\ge0$.
    \item The derivation of (\ref{eq:S_n^Y}) assumed that the linear combination was convex, but it need not be. Provided that $w_1+w_2=1$ and $y(n)$ is still a valid probability mass function, then we can allow for $w_1,w_2<0$. This is useful for analyzing perturbations of distributions, but also for constructing certain complicated distributions.
\end{itemize}

\subsection{Some Examples}
\subsubsection{Example: Mixture of two point masses}
Let $x_1(n) = \delta_{M_1,n}$ and $x_2(n) = \delta_{M_2,n}$ for $M_1,M_2\in\naturals$. Then $X_1(n) = \one_{[M_1,\infty)}(n)$ and $\expectation[\Xc_1] = M_1$, and similar for $\Xc_2$. Furthermore, we know that these purely deterministic FP distributions do not permit beneficial restart, which means that $S_n^{\Xc_1}\ge0$ for all $n\ge0$, and we can confirm that by computing
\begin{align*}
    S_n^{\Xc_1} & = \sum_{k=0}^n 1 - X_1(k) - \expectation[\Xc_1]x_1(k) \\
    & = \sum_{k=0}^n \one_{[0,M_1)}(n) - M_1\delta_{M_1,n} \\
    & = \begin{cases}
       n + 1 & \text{ if } 0 \le n < M_1 \\
       0 & \text{ if } M_1 \le n
    \end{cases} \ge 0 \text{ for all } n\ge0.
\end{align*}

Now, let's suppose that $M_2 > M_1$, and define $y(n) \coloneqq w_1x_1(n) + w_2x_2(n)$ for some appropriate $w_1,w_2$. By (\ref{eq:S_n^Y}), we have 
\begin{equation*}
    S_n^\Yc = w_1S_n^{\Xc_1} + w_2S_n^{\Xc_2} + w_1w_2(M_2-M_1)(X_2(n)-X_1(n)).
\end{equation*}
The difference in cumulative partial sums reduces to 
\begin{equation*}
    X_2(n) - X_1(n) = \begin{cases} 0 & \text{ for } 0\le n < M_1 \\ -1 & \text{ for } M_1 \le n < M_2 \\ 0 & \text{ for } M_2 \le n \end{cases}.
\end{equation*}
Thus the partial sums collapse to 
\begin{equation*}
    S_n^\Yc =  
    \begin{cases}
       n + 1 & \text{ for } 0 \le n < M_1 \\
       w_2\left( n + 1 - w_1(M_2 - M_1)\right) & \text{ for } M_1 \le n < M_2 \\
       0 & \text{ for } M_2 \le n
    \end{cases}.
\end{equation*}
The partial sums are clearly non-negative for $n$ smaller than $M_1$ and greater than or equal to $M_2$, so the only place we need look for beneficial restart is for $M_1 \le n < M_2$. In this region, $S_n^\Yc<0$ iff
\begin{align*}
    0 & > w_2\left( n + 1 - w_1(M_2 - M_1)\right) \\
     n & < w_1(M_2 - M_1) - 1,
\end{align*}
which is only possible when $M_1 < w_1(M_2 - M_1) - 1$, or equivalently $M_2 > \frac{(1+w_1)M_1 + 1}{w_1}$, which eliminates the possibility of $M_1$ and $M_2$ adjacent. This result clearly coincides with the one in Section \ref{ex:two_point_mixture}. 

\subsubsection{Example: Mixture of two geometric distributions on $\naturals$}
Let $x_1(n) = \rho(1-\rho)^n$ and $x_2(n) = \gamma(1-\gamma)^n$ with $\gamma \ge \rho$ so $\expectation[\Xc_1] \ge \expectation[\Xc_2]$. We note that $X_1(n) = 1 - \left(1-\rho\right)^{n+1}$ and $\expectation[\Xc_1] = \frac{1-\rho}{\rho}$ with similar results for $\Xc_2$. Computing the partial sums yields 
\begin{equation*}
    S_n^{\Xc_1} = \sum_{k=0}^n 1 - \left(1 - (1-\rho)^{n+1}\right) - \frac{1-\rho}{\rho}\cdot\rho(1-\rho)^n = 0.
\end{equation*}
Just as previously noted, the geometric distribution starting at 0 is the threshold distribution, with its partial sums equal to 0 for all $n \ge 0$. This simplifies (\ref{eq:S_n^Y}) to 
\begin{align*}
    S_n^\Yc & = w_1w_2(\expectation[\Xc_1] - \expectation[\Xc_2])(X_1(n) - X_2(n)) \\
    & = \frac{w_1w_2}{\rho\gamma}(\gamma - \rho)((1-\gamma)^{n+1} - (1-\rho)^{n+1}).
\end{align*}
From the above, it is clear that $S_n^\Yc \le 0$ for all $n\ge0$ with equality only when $\gamma = \rho$ (and in that case, clearly $\Yc$ is simply geometric and we know its partial sums to be $0$). Thus, for $\gamma\neq\rho$, we have that beneficial restart is possible -- in particular, sharp restart is beneficial for all $N\ge1$.

\subsubsection{Example: Perturbation by geometric distribution}
Let FP time $\Xc_1$ have some arbitrary PMF $x_1(n)$, and $\Xc_2$ be distributed geometrically on $\naturals$ with rate parameter $\gamma$. Since the restart sequence for $\Xc_2$ is identically $0$, we can easily compute the partial sums for $\Yc$.
\begin{equation*}
    S_n^\Yc = w_1S_n^{\Xc_1} + w_1w_2(\expectation[\Xc_1] - \frac{1-\gamma}{\gamma})(X_1(n) + (1-\gamma)^{n+1} - 1)
\end{equation*}
In the case where $\gamma$ is selected so that $\expectation[\Xc_2]=\expectation[\Xc_1]$, the second term is clearly 0, and we have simply $S_n^\Yc = w_1S_n^{\Xc_1}$. In other words, for any FP $\Xc_1$ with finite mean, there exists $\gamma \in (0,1)$ such that a perturbation by the geometric distribution with rate $\gamma$ only scales the restart sequence by $w_1$, and thus does not affect whether restart is beneficial at all.


\section{Conclusion}

Using the construct of the restart sequence allows us to understand the nature of FPUR processes. In particular, it allows us to see that any FP process with a log-concave PMF will not permit beneficial restart, and any process with a log-convex PMF will, except one which is geometrically distributed, which isn't affected by restart at all. Additionally, if the tail of the PMF decays more quickly than geometrically, late restarts won't be beneficial. On the other hand, if the PMF has a heavy tail, then all late restarts will be beneficial.
Much investigation remains to be done. Can this tool be extended to the continuous time case? Can the restart sequence provide useful information about FP processes with highly fragmented supports? Additionally, exploring the kinds of perturbations to FP time distributions that can be analyzed using the restart sequence can be explored moving forward.

\nocite{*}

\bibliography{references}

\end{document}